\documentclass[12pt,reqno]{amsart}
\usepackage{amssymb,amsmath,amsthm}
\usepackage{graphicx}
\usepackage{subcaption}
\usepackage{fullpage}
\usepackage{scrextend}
\usepackage{siunitx}
\usepackage{enumerate}
\usepackage{tikz,calc}
\usepackage{tikz-cd}
\usetikzlibrary{automata,positioning}
\usepackage{comment}
\usetikzlibrary{calc}
\usepackage{epsfig,graphicx}
\usepackage{listings}
\usepackage{enumerate}
\usepackage{float}
\usepackage{bbm}
\usepackage[colorlinks=true, pdfstartview=FitH, linkcolor=blue, citecolor=blue, urlcolor=blue]{hyperref}

\newcommand{\Mod}[1]{\ (\text{mod}\ #1)}
\renewcommand{\Mod}[1]{{\ifmmode\text{\rm\ (mod~$#1$)}\else\discretionary{}{}{\hbox{ }}\rm(mod~$#1$)\fi}}

\newtheorem{theorem}{Theorem}[section]

\newtheorem{lemma}[theorem]{Lemma}
\newtheorem{cor}[theorem]{Corollary}
\newtheorem{claim}[theorem]{Claim}

\theoremstyle{definition} 

\numberwithin{theorem}{section}

\begin{document}

\title{Additive structure in convex translates }
\author{Gabriel Currier}
\address{Department of Mathematics \\ University of British Columbia \\ Room 121, 1984 Mathematics Road \\ Vancouver, BC, Canada V6T 1Z2}
\email{currierg@math.ubc.ca}
 \author{Jozsef Solymosi}
\address{Department of Mathematics \\ University of British Columbia \\ Room 121, 1984 Mathematics Road \\ Vancouver, BC, Canada V6T 1Z2}
\email{solymosi@math.ubc.ca}
\author{Ethan Patrick White}
\address{Department of Mathematics \\ University of British Columbia \\ Room 121, 1984 Mathematics Road \\ Vancouver, BC, Canada V6T 1Z2}
\email{epwhite@math.ubc.ca}

\maketitle

\begin{abstract} Let $\mathcal{P}$ be a set of points in the plane, and $\mathcal{S}$ a strictly convex set of points. In this note, we show that if $\mathcal{P}$ contains many translates of $\mathcal{S}$, then these translates must come from a generalized arithmetic progression of low dimension. We also discuss an application to the unit distance conjecture.


\end{abstract}

\section{Introduction} 

Suppose we have a set of $n$ points and $n$ translates of a strictly convex curve. In general, it is known that such collections of points and curves can have only $\lesssim n^{4/3}$ incidences\footnote{Throughout the paper we use the notation $A\gtrsim B$ to mean there exists a universal constant $C>0$ such that $A \geq CB$}~\cite{PASH}. Most of the convex curves cannot achieve this bound \cite{Ma}, but there are examples where they can~\cite{Ja,SSz,VAL}.\\

The standard constructions for these examples proceed as follows: we construct a curve that has $n^{1/3}$ incidences with a grid, and then take translates of that curve by the set of vectors determined by that grid. In this situation, all of our incidences come from $n^{1/3}$ fixed points along a curve, and our translates come from a generalized arithmetic progression (in this case, of dimension $2$).\\

Our main result shows that the first of these conditions implies the second. That is, if our incidences come from $n^{1/3}$ fixed points along a curve, then our translates \emph{must} come from a generalized arithmetic progression of bounded dimension.

\begin{theorem}\label{MT} Let $\mathcal{P}$ be a set of $n$ points in $\mathbb{R}^2$, $\mathcal{S}$ be a strictly convex set of $\lesssim n^{1/3}$ points, and $\mathcal{T}$ be a set of $\lesssim n$ vectors in $\mathbb{R}^2$. If $\mathcal{S} + \mathcal{T}$ intersects $\mathcal{P}$ in $\gtrsim n^{4/3}$ points, then there is a subset $\mathcal{T}'\subset \mathcal{T}$ of size $\gtrsim n$ contained in a generalized arithmetic progression of dimension $ \lesssim 1$ and size $\lesssim n$.

\end{theorem}

Our main application is to the \emph{unit distance problem}, asked by Erd\H{o}s in 1946~\cite{ERD}. The problem is to estimate the maximum number of pairs from a set of $n$ points in the plane that are distance one apart. The best known lower bound of $n^{1 + c/\log \log n}$ is due to Erd\H{o}s and relies on finding a number that can be written as a sum of two squares in many ways, followed by a corresponding scaling of a finite integer grid. On the other hand, the best upper bound of $cn^{4/3}$ is due to Spencer, Szemer\'edi, and Trotter~\cite{SST}. For an excellent survey of this problem, see~\cite{Szem}. Part of the difficulty in improving the upper bound is that, as mentioned before, $cn^{4/3}$ unit distances can occur under a different strictly convex norm~\cite{SSz,VAL}.\\

Several authors have studied the maximum number of unit distances with an additional restrictive property satisfied by grids. For example, Schwartz, Solymosi, and de Zeeuw showed that the number of pairs of points that determine a rational slope and are unit distance apart is $n^{1 + 6/\sqrt{\log n}}$~\cite{SSD1}. Schwartz generalized this result with the restriction that the unit vectors determined belong to a low-rank multiplicative group when embedding in the complex plane~\cite{SCH}. If we instead fix a set of $k$ unit vectors and ask for the maximum number of times a vector from our chosen set can be determined, the answer is $kn-\Theta(\sqrt{n})$, proved by Brass~\cite{BRASS}. The configurations of points achieving the maximum here are also lattices. \\

We extend Brass's result to allow unit distances from a set of unit vectors that grows with the size of the pointset. Our result is a structure theorem, showing that if the number of unit distances achieved is maximum, then a large portion of the pointset is contained in a generalized arithmetic progression. 


\begin{cor}\label{MC} Let $\mathcal{P}$ be a set of $n$ points in the plane, and $U \subset \mathbb{R}^2$ be a set of unit length vectors where $|U| \lesssim n^{1/3}$. If $\# \{ (x,y) \in \mathcal{P} \colon x-y \in U\} \gtrsim n^{4/3}$ then $\gtrsim n$ points of $\mathcal{P}$ are contained in a generalized arithmetic progression of dimension $\lesssim 1$ and size $\lesssim n$.

\end{cor}

Corollary~\ref{MC} is obtained from Theorem~\ref{MT} by representing the set $U$ of unit vectors as a set of points on a unit circle, a strictly convex curve.

\section{Preliminaries}

We will need several standard tools from additive combinatorics and discrete geometry. The first is a variant of the Szemer\' edi-Trotter theorem \cite{SZT} applying to a slightly more general class of curves (see, e.g. \cite{PASH}). We say that a collection of simple curves $C$ are \emph{pseudo-lines} if any two curves from $C$ intersect in at most one point.

\begin{theorem}[Szemer\'edi-Trotter]\label{szt}
The number of incidences between $n$ points and $m$ pesudo-lines is $\lesssim n^{2/3}m^{2/3} + n + m$.
\end{theorem}

Next, we will need the following consequence of the triangle removal lemma of Ruzsa and Szemer\'edi \cite{RuSze}. The best-known quantitative bound on the triangle removal lemma is due to Fox \cite{Fox}.

\begin{theorem}\label{trirem}
Let $G$ be a graph on $n$ vertices, and suppose $G$ contains $\gtrsim n^2$ edge-disjoint triangles. Then $G$ contains $\gtrsim n^3$ triangles.
\end{theorem}
The remaining two theorems are frequently used to show structure in subsets of additive groups. Both are well-known tools in additive combinatorics. A comprehensive treatment of these results can be found in, e.g. \cite{TV}.

\begin{theorem}[Balog-Szemer\' edi-Gowers \cite{BSG1,BSG2}]\label{bsg}
Suppose $G$ is an abelian group, $A \subset G$ is finite, and $H$ is a graph with vertex set $A$ and $\gtrsim |A|^2$ edges. Then, if $|A -_H A| \lesssim |A|$ there must exist $A' \subset A$ such that $|A'| \gtrsim |A|$ and $|A' - A'| \lesssim |A'|$
\end{theorem}

\begin{theorem}[Freiman-Ruzsa \cite{FREIMAN,RUZSA}]\label{frei}
Suppose $G$ is an abelian group, $A \subset G$ is finite, and $|A-A| \lesssim |A|$. Then, $A$ is contained in a generalized arithmetic progression of size $\lesssim |A|$ and dimension $\lesssim 1$.
\end{theorem}

\section{Many copies of a convex pointset}




Before beginning a proof of Theorem~\ref{MT} we make a reduction to show that $\mathcal{S}$ can be assumed to lie on a convex curve with the following characteristics. We say a strictly convex curve $F$ is \emph{nice} if

\begin{enumerate}[(1)] 

\item Any two translates of $F$ intersect exactly once, unless one is a vertical shift of the other, in which case they do not intersect at all.

\item $F$ is $x$-monotone; every vertical line intersects $F$ exactly once.

\item For any pair of points in $\mathbb{R}^2$ not on a vertical line there is a unique translate of $F$ that passes through them both.


\end{enumerate}

Let $\mathcal{P}$, $\mathcal{S}$, and $\mathcal{T}$ be as in the hypothesis of Theorem~\ref{MT}. We can partition $\mathcal{S} = \cup_{i=1}^3 \mathcal{S}_i$ into three sets such that each $\mathcal{S}_i$ after an appropriate rotation is contained in a nice curve. At least one $\mathcal{S}_i+\mathcal{T}$ intersects $\mathcal{P}$ in $\gtrsim n^{4/3}$ points. It follows that we can assume $\mathcal{S}$ lies on a nice curve $F$. The benefit of this reduction is that it is easy to obtain the following cutting theorem about nice curves. 


\begin{theorem}\label{cuttheo} Let $\mathcal{F}$ be a set of $n$ translates of a nice curve $F$. For any choice of parameter $11\leq r<n$, there exists a decomposition of the plane into at most $20r^2$ cells, where the boundary of each cell consists of a union of at most two arcs from $\mathcal{F}$ and at most two vertical line segments, such that at most $n/r$ curves in $\mathcal{F}$ intersect the interior of any cell.

\end{theorem}

We give a proof of Theorem~\ref{cuttheo} in Section~\ref{cutsec} by adapting an argument of Matou\v{s}ek for the equivalent theorem on lines~\cite[\S 4.7]{MAT}.

\begin{proof}[Proof of Theorem~\ref{MT}] Let $\mathcal{F} = F + \mathcal{T}$ be the family of translates of $F$. 

\begin{itemize} 
\item A \emph{good curve} from $\mathcal{F}$ contains $\gtrsim n^{1/3}$ points from $\mathcal{P}$. We let $\mathcal{F}' \subset \mathcal{F}$ be the set of good curves and $\mathcal{T}'\subset \mathcal{T}$ be the corresponding set of translating vectors.

\end{itemize} 

We note that since $\mathcal{S} + \mathcal{T}$ intersects $\mathcal{P}$ in $\gtrsim n^{4/3}$ points, we must have that $|\mathcal{F}'| \gtrsim |\mathcal{F}| \gtrsim n$. Throughout our proof we will choose constants $C_i$, $i=1,2,3,4$. We begin by applying Theorem~\ref{cuttheo} on our set of curves $\mathcal{F}'$, with the choice $r = C_1n^{1/3}$. We obtain a cutting of $\mathbb{R}^2$ into at most $20C_1^2n^{2/3}$ cells such that each cell is entered by at most $\frac{|\mathcal{F}'|}{C_1n^{1/3}}$ curves from $\mathcal{F}'$. Our cutting is composed of $\lesssim n^{2/3}$ pseudolines, as shown in the proof of Theorem~\ref{cuttheo}. Hence the number of incidences between $\mathcal{S}+\mathcal{T}$ and $\mathcal{P}$ is $\lesssim n^{10/9}$ by Theorem~\ref{szt}. For the rest of the proof, we will ignore points that are contained in the cell-boundaries of our cutting. Furthermore, from now on when we say a point is ``in'' a cell, we mean in the interior. \\

We make the following two definitions to aid with showing the typical behaviour of points and curves in each cell.



\begin{itemize}

    \item Let $F_0 \in \mathcal{F}'$ be a good curve and $\mathcal{S}_0$ be the set of points from $\mathcal{S}$ on $F_0$. A \emph{good triple} is a set of three points $x_0,y_0,z_0 \in \mathcal{P}\cap \mathcal{S}_0$ ordered left to right along $F_0$ such that: i) the points are in the same cell; ii) no other points from $\mathcal{P}\cap \mathcal{S}_0$ lie on $F_0$ between $x_0$ and $z_0$; and iii) if $x,y,z \in \mathcal{S}$ are the points corresponding to $x_0,y_0,z_0$ then the number of points between $x$ and $z$ on $F \cap \mathcal{S}$ is at most $C_2$.
    

    \item A \emph{good cell} has $\geq C_3n^{2/3}$ distinct good curves containing at least one good triple, and $\le C_4n^{1/3}$ points in it. We note that since a good cell clearly has $\geq C_3n^{2/3}$ good triples, and since a pair of points can belong to at most two good triples, a good cell necessarily also has $\gtrsim{n^{1/3}}$ points.

\end{itemize}



\begin{quote}

\begin{claim}\label{gtclaim} There are $\gtrsim n^{4/3}$ good triples. 

\end{claim}


\begin{proof} Pick an element $F_0 \in \mathcal{F}'$ and let $\mathcal{S}_0$ be the points from $\mathcal{S}$ on $F_0$. Divide the points in $\mathcal{S}_0$ into consecutive intervals of length $C_2$ from left to right. Suppose that at least $|\mathcal{S}_0|/(C_2+1)$ of these intervals have at most $2$ points from $\mathcal{P}$ in them. Then, the total number of incidences of $\mathcal{P}$ with $\mathcal{S}_0$ is
\[\frac{2|\mathcal{S}_0|}{C_2+1} + \left(\frac{|\mathcal{S}_0|}{C_2} - \frac{|\mathcal{S}_0|}{C_2+1}\right) C_2 = \frac{3}{C_2+1}|\mathcal{S}_0|.\]
If $C_2$ is chosen to be sufficiently large, then this is a contradiction, since $|\mathcal{S}_0|=|\mathcal{S}| \lesssim n^{1/3}$ and each good curve has $\gtrsim n^{1/3}$ points on it by definition. Thus, there must be at least $\left(\frac{1}{C_2} - \frac{1}{C_2+1}\right)|\mathcal{S}_0|$ intervals of length $C_2$ containing at least $3$ points. Within each such interval, there is at least one triple meeting requirements ii) and iii) of the good triple definition. Therefore there are $\gtrsim n^{1/3}|\mathcal{F}'|$ good triples before cutting.\\


Now, by our cell decomposition from Theorem \ref{cuttheo}, we know that we have $\le 20C_1^2n^{2/3}$ cells, and at most $\frac{|\mathcal{F}'|}{C_1n^{1/3}}$ curves from $\mathcal{F}'$ intersecting each cell. Thus, there are $\le |\mathcal{F}'|20C_1n^{1/3}$ curve-cellwall incidences, and each can destroy at most two good triples. Choosing $C_1$ to be sufficiently small shows that there are $\gtrsim n^{1/3}|\mathcal{F}'|$, since $|\mathcal{F}'|\gtrsim n$ the claim follows.



\end{proof}

\end{quote}

\begin{quote}

\begin{claim}\label{gccount} There are $\gtrsim n^{2/3}$ good cells. 

\end{claim}

\begin{proof} Let the cells have index set $I$, and $p_i$ be the number of points in cell $i\in I$ from $\mathcal{P}$, and $f_i$ the number of curves contributing at least one good triple to cell $i$. Each good triple corresponds to a curve-point incidence. Therefore by Szemer\'edi-Trotter (Theorem \ref{szt}) and Claim \ref{gtclaim} we have
\[ \sum_{i} \big((p_if_i)^{2/3}+p_i+f_i\big)  \gtrsim n^{4/3}. \]

Let $J \subset I$ be an index set. Since $\sum_i p_i \lesssim n$ and using H\" older's inequality
\begin{equation}\label{holderbound}
\sum_{j \in J} \big((p_if_i)^{2/3}+p_i+f_i \big) \lesssim  (\max_{j \in J} f_j)^{2/3}n^{2/3}|J|^{1/3}+n+(\max_{j\in J} f_j)|J|.
\end{equation}

Let $J_1$ be the indices of cells where $f_i \le C_3n^{2/3}$, and $J_2$ be the indices of cells with more than $C_4n^{1/3}$ points. For cells in $J_1$ we have $\max_{j\in J_1} f_j \le C_3n^{2/3}$, and so~\eqref{holderbound} is $\lesssim$ $C_3n^{4/3}$ for $J = J_1$. Furthermore, since we have $\leq n$ points, we see $|J_2| \leq (1/C_4)n^{2/3}$. Since $f_i \lesssim n^{2/3}$ for all $i\in I$ (by our cutting), \eqref{holderbound} is $\lesssim \frac{1}{C_4^{1/3}}n^{2/3}$ when $J=J_2$. Letting $C_3$ be sufficiently small and $C_4$ be sufficiently large, we see then 
$$(\max_{j \notin J_1\cup J_2} f_j)^{2/3}n^{2/3}|I \setminus (J_1 \cup J_2)|^{1/3}+n+(\max_{j \notin J_1\cup J_2} f_j)|I \setminus (J_1 \cup J_2)| \gtrsim n^{4/3}$$
From which we immediately conclude $|I \setminus (J_1 \cup J_2)| \gtrsim n^{2/3}$. This completes the claim.



\end{proof}
\end{quote}





For a single cell $D$, define the graph $G_D$, drawn in the plane as follows. Let the vertices be the points from $\mathcal{P} \cap (\mathcal{S}+\mathcal{T})$ in $D$. For each curve $F_0 \in \mathcal{F}'$ with least one good triple in $D$, choose exactly one such triple $\{x,y,z\}$, and let $\{x,y\}$, $\{x,z\}$ and $\{y,z\}$ be edges in $G_D$, where the edges are drawn along $F_0$. For any path of length 3 $\{e_1,e_2,e_3\}$ in $G_D$ we say the path is \emph{self-intersecting} if the curves corresponding to $e_1$ and $e_3$ intersect in $D$.

\begin{quote}
\begin{claim}\label{claimSP3} In a good cell there are $\gtrsim n^{4/3}$ self intersecting $P_3$'s. 

\end{claim}
\begin{proof} Let $D$ be a good cell, and $G_D$ be as above. Note that since $D$ is a good cell, $G_D$ has $\lesssim n^{1/3} $ vertices. From Claim~\ref{gccount} we know that $G_D$ contains $\gtrsim n^{2/3}$ edge-disjoint triangles, so by Theorem \ref{trirem} $G_D$ contains $\gtrsim n$ triangles. Consider two triangles $T_1 = \{e_1,e_2,e_3\},T_2= \{e_1',e_2,e_3'\}$ in $G_D$ that share an edge. Let $v \in T_1$, $v'\in T_2$ be the vertices not on the edge $e_2$. If $v,v'$ are on the same side of the curve in $\mathcal{F}'$ that contains $e_2$, then a self-intersecting $P_3$ will be formed since curves can intersect at most once. See Figure~\ref{SP3} for illustration. \\

We can now complete the claim with a convexity argument. For each edge $e$ in $G_D$, let $t(e)$ be the number of triangles containing $e$. Then by the above, the number of intersecting $P_3$'s is at least 
\[ \sum_{e \in E(G_D)} \binom{\lceil t(e)/2 \rceil}{2} \gtrsim \frac{1}{|E(G_D)|}\bigg(\sum_{e \in E(G_D)} t(e)\bigg)^2 \gtrsim n^{4/3}.\]


\end{proof}

\end{quote}

\begin{figure}[ht]
\begin{center}
\begin{tikzpicture}[scale=1]

\def\x{-1}

\filldraw

(-3,0) circle (2pt)
(3,0) circle (2pt)
(-1.5,1.5) circle(2pt)
(-1,3) circle(2pt);

\draw
(-3,0) to [out = 20, in = 160] (3,0)
(-3,0) to [out = 55, in = 215] (-1.5,1.5)
(-3,0) to [out = 80, in = 210] (-1,3)
(-1,3) to [out = 350, in = 120] (3,0);

\draw[dotted]
(-1.5,1.5) to [out = 35, in = 180] (1.94,2.57);

\draw
(0,0.25) node{$e_2$}
(-2.5,2) node{$e_1$}
(-2,0.8) node{$e_1'$}
(1.7,2.1) node{$e_3$};

\end{tikzpicture}
\end{center}
\caption{Self-intersecting $P_3$: $e_1',e_2,e_3$.}
\label{SP3}
\end{figure}
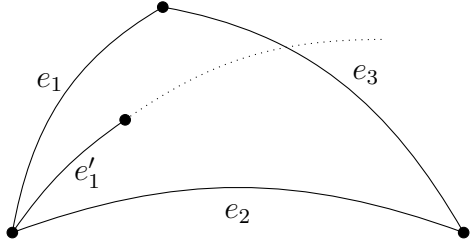

Now we put together the pieces. Define a graph $\mathcal{H}$ with vertex set $\mathcal{F}'$. For two curves $F_1,F_2 \in \mathcal{F}'$, we make $F_1\sim F_2$ an edge if there are edges $e_1 \in F_1$ and $e_3 \in F_2$ and a third edge $e_2$ such that $\{e_1,e_2,e_3\}$ are all contained in the same cell, and form a self-intersecting $P_3$. We now count the edges in $\mathcal{H}$. The total number of self-intersecting $P_3$'s is $\gtrsim n^2$ by Claim~\ref{claimSP3}. Curves in $\mathcal{F}'$ intersect at most once, and within each cell $D$ we select only 3 arcs from a single curve to become edges of $G_D$. Furthermore, a self-intersecting $P_3$ can be counted in at most one cell, so at most $9$ self-intersecting $P_3$'s correspond to the same edge, and so $\mathcal{H}$ has $\gtrsim n^2$ edges.  \\


The edges in each cell graph $G_D$ correspond to $\lesssim n^{1/3}$ fixed arcs along our curve $F$. Supposing $F_1 \sim F_2$ in $\mathcal{H}$, there exist arcs $e_1,e_2$ along $F_1$ and $F_2$ (respectively) and a third arc $e_0$ such that $e_1,e_0,e_2$ is a self-intersecting $P_3$ in some good cell. Since the $e_i$ each come from the $\lesssim n^{1/3}$ fixed arcs, and the vector difference between $F_1$ and $F_2$ is determined by $e_0,e_1,e_2$, we see $|\mathcal{T}'-_\mathcal{H} \mathcal{T}'| \lesssim n$. By Theorem \ref{bsg} we obtain $\mathcal{T}''\subset \mathcal{T}'$ of size $\gtrsim n$ such that $|\mathcal{T}''-\mathcal{T}''| \lesssim n$. Finally, by Theorem \ref{frei} we conclude that $\mathcal{T}''$ is contained in a generalized arithmetic progression of dimension $\lesssim 1$ and size $\lesssim n$.





\end{proof}

\section{Cutting theorem for strictly convex curves}\label{cutsec}








In this section we prove Theorem~\ref{cuttheo} by adapting an argument of Matou\v{s}ek for the analogous theorem on lines~\cite[\S 4.7]{MAT}. Let $n,F,$ and $\mathcal{F}$ be as in the statement of Theorem~\ref{cuttheo}. By applying a small perturbation to $\mathcal{F}$, we may assume that $\mathcal{F}$ is in general position: all pairs of curves will intersect and there will be no points of triple intersection. Our actual cutting is then given by a limit of these cuttings as the size of the perturbation goes to zero.\\

We will call the points of intersection between curves in $\mathcal{F}$ vertices, and the vertex-free open arc-segments of curves in $\mathcal{F}$ \emph{edges}. Note that some edges will have finite length, and others will be unbounded towards the right or left.\\


 
 

 
The \emph{level} of a point $x \in \mathbb{R}^2$ is the number of curves in $\mathcal{F}$ lying strictly below $x$. Note that the points on any edge have the same level. For each $0 \leq k \leq n$, define the \emph{level $k$} of $\mathcal{F}$ as the set of edges with level $k$, along with their endpoints. Denote the set of edges in level $k$ by $E_k$.\\

Fix a pair of points $x,y \in \mathbb{R}^2$ and let $F_{x,y}$ denote the arc-segment between $x,y$ of the unique translate of $F$ containing both $x$ and $y$. Let the edges in $E_k$ be $e_0,e_1,\ldots,e_t$, and note that $e_0$ and $e_t$ are unbounded to the left and right, respectively. Choose arbitrary points $p_i \in e_i$ for $0 \leq i \leq t$, and fix a parameter $q \geq 2$. Define the $q$\emph{-simplification of level $k$} as union of arcs 
\[ F_{p_0,p_q}, F_{p_q,p_{2q}} ,\ldots, F_{p_{\lfloor (t-1)/q \rfloor q} p_t},\]
in addition to the part of $e_0$ to the left of $p_0$, and the part of $e_t$ to the right of $p_t$. Note that $q$-simplification of level $k$ is an $x$-monotone curve, and that it consists of at most $t/q+3$ arc-segments.


\begin{lemma}\label{simpprop} \

\begin{enumerate}[(i)]

\item The portion $\Pi$ of the level $k$ between $p_j$ and $p_{j+q}$ is intersected by at most $q+1$ translates in $\mathcal{F}$.

\item The arc $F_{p_j,p_{j+q}}$ is intersected by at most $q+1$ translates in $\mathcal{F}$.

\item The $q$-simplification of level $k$ is contained in the strip between levels $k-\lceil q/2 \rceil$ and $k+\lceil q/2 \rceil$.

\end{enumerate}

\end{lemma}

\begin{proof} (i) any curve of $\mathcal{F}$ intersecting $\Pi$ must belong to the curves comprising $\Pi$, otherwise the level would change along $\Pi$. There are at most $q+1$ distinct curves comprising $\Pi$. (ii) The union of $F_{p_j,p_{j+q}}$ and $\Pi$ divides $\mathbb{R}^2$ into a series of bounded cells. If an element of $\mathcal{F}$ intersects $F_{p_j,p_{j+q}}$, then in order to leave the cell it must intersect $\Pi$ as well. Since at most $q+1$ curves from $\mathcal{F}$ intersect $\Pi$, we are done. (iii) We consider how high or low the level can become moving left to right along $F_{p_j,p_{j+q}}$. The level begins and ends at $k$. Each change of level must be accompanied by a curve in $\mathcal{F}$ intersecting $F_{p_j,p_{j+q}}$. Hence to move from level $k$ to $k\pm i$ and back to $k$, at least $2i$ curves of $\mathcal{F}$ must cross $F_{p_j,p_{j+q}}$. It follows that $2i \leq q+1$.

\end{proof}

\begin{proof}[Proof of Theorem~\ref{cuttheo}] If $r \geq n/4$, then we can form the required cutting by choosing all curves in $\mathcal{F}$. We will now assume $11 \leq r<n/4$. Set $q = \lfloor n/3r \rfloor-1 $. The total number of edges in $\mathcal{F}$ is $n^2$. We may find $0 \leq i \leq q-1$ such that the number of edges in $\cup_j E_{jq+i}$ is at most $n^2/q$. Let $P_j$ be the $q$-simplification of $E_{jq+i}$ for $0 \leq j \leq (n-1)/q$. Let $m_j$ denote the number of edges in $E_{jq+i}$. Then the total number of edges among all $P_j$ is at most 
\[ \sum_{0 \leq j \leq (n-1)/q} \left( m_j/q + 3 \right) \leq \frac{n^2}{q^2}+ \frac{3n}{q}.\]


Note that no two $q$-simplifications intersect. If they did, then a vertex from some $P_j$ would lie above $P_{j+1}$, but all vertices of $P_j$ have level $jq+i$, which contradicts Lemma~\ref{simpprop} (iii). \\

The union of the $P_j$ for $0 \leq j \leq (n-1)/q$ will form our decomposition, along with the additional vertical line segments. For every vertex in $P_j$, extend vertical lines up and down until they reach $P_{j+1}$ and $P_{j-1}$. Each vertical line segment creates an additional cell and so we have obtained a partition of the plane into at most $2n^2/q^2+7n/q+2< 20r^2$ cells. \\

Now we verify that at most $n/r$ curves of $\mathcal{F}$ enter any single cell. A typical cell $D$ will be between a pieces of $P_j$ and $P_{j+1}$ for some $0 \leq j \leq (n-1)/q$ and also bounded by two vertical segments. By Lemma~\ref{simpprop} (iii) we know $D$ lies between levels $jq+i-\lceil q/2 \rceil$ and $(j+1)q+i+\lceil q/2 \rceil$. Hence the vertical line segments bounding the set can be intersected by at most $2q+1$ curves in $\mathcal{F}$. The upper and lower boundaries of $D$ are arcs as in Lemma~\ref{simpprop} (ii) and so are intersected by at most $q+1$ curves in $\mathcal{F}$. Since a curve entering $D$ must intersect the boundary twice, we conclude that at most $3q+2$ curves of $\mathcal{F}$ intersect $D$. There are also atypical cells below $P_0$, above $P_{\lfloor (n-1)/q \rfloor}$, or bounded by a single vertical segment, but such cells are easily verified to have less than $3q$ curves of $\mathcal{F}$ intersecting their interior. Since $3q+2<n/r$ we have proved the result.

\end{proof}


\end{document}